\documentclass[a4paper, oneside, 11pt]{article}
\usepackage{amsmath,amsfonts,amssymb,amscd,amsthm}
\usepackage{enumerate}
\usepackage[all]{xy}
\usepackage{minibox}
\usepackage{hyperref}
\usepackage{url}
\usepackage{bookmark}

\newtheorem{thm}{Theorem}[section]
\newtheorem{prop}[thm]{Proposition}
\newtheorem{lem}[thm]{Lemma}
\newtheorem{cor}[thm]{Corollary}
\newtheorem{fact}[thm]{Fact}
\newtheorem{conj}[thm]{Conjecture}

\theoremstyle{remark}
\newtheorem{rem}[thm]{Remark}
\newtheorem*{acknow}{Acknowledgements}

\newcommand{\N}{\mathbb{N}}
\newcommand{\Z}{\mathbb{Z}}

\newcommand{\R}{\mathbb{R}}

\newcommand{\K}{\mathbb{K}}

\newcommand{\LL}{\mathcal{L}}

\newcommand{\g}{\mathfrak{g}}
\newcommand{\h}{\mathfrak{h}}
\renewcommand{\k}{\mathfrak{k}}
\renewcommand{\l}{\mathfrak{l}}

\newcommand{\im}{\operatorname{image}}
\newcommand{\sgn}{\operatorname{sgn}}
\newcommand{\rank}{\operatorname{rank}}
\newcommand{\rest}{\operatorname{rest}}

\newcommand{\w}{\wedge}
\newcommand{\ox}{\otimes}
\newcommand{\bs}{\backslash}
\newcommand{\bl}{\bullet}
\newcommand{\shift}{\widetilde}
\newcommand{\ep}{\varepsilon}
\newcommand{\injto}{\hookrightarrow}
\newcommand{\surjto}{\twoheadrightarrow}
\newcommand{\simto}{\xrightarrow{\sim}}
\renewcommand{\geq}{\geqslant}
\renewcommand{\leq}{\leqslant}

\newcommand{\bysame}{---------}

\begin{document}

\title{Proof of Kobayashi's rank conjecture on Clifford--Klein forms}
\author{Yosuke Morita}
\date{}

\maketitle

\begin{abstract}
T. Kobayashi conjectured in the 36th Geometry Symposium
in Japan (1989) that a homogeneous space $G/H$ of reductive type
does not admit a compact Clifford--Klein form if
$\rank G - \rank K < \rank H - \rank K_H$.
We solve this conjecture affirmatively.
We apply a cohomological obstruction to the existence of
compact Clifford--Klein forms proved previously by the author,
and use the Sullivan model for a reductive pair due to
Cartan--Chevalley--Koszul--Weil.
\end{abstract}

\maketitle

\section{Introduction}\label{rank:sect:intro}

A Clifford--Klein form of a homogeneous space $G/H$
is a quotient space
$\Gamma \bs G/H$, where $\Gamma$ is a discrete subgroup of $G$
acting properly and freely on $G/H$.
It is a typical example of a manifold locally modelled on $G/H$,
i.e.\ a manifold obtained by patching open sets of $G/H$
by left translations by elements of $G$.
Since the initial work \cite{Kob89Ann} by T. Kobayashi,
the existence problem of compact Clifford--Klein forms
has been studied by various methods (e.g.\ \cite{Kob-Ono90},
\cite{Kob92Duke}, \cite{Zim94}, \cite{Ben96}, \cite{Mar97}).

In this paper, we solve a conjecture on the nonexistence of
compact Clifford--Klein forms, posed by Kobayashi~\cite{Kob89Geom}
in 1989, affirmatively.
Recall that a homogeneous space $G/H$ is called of reductive type
if $G$ is a linear reductive Lie group with Cartan involution
$\theta$ and $H$ is a closed subgroup of $G$
with finitely many connected components such that $\theta(H) = H$.
We write $K$ and $K_H$
for the corresponding maximal compact subgroups of $G$ and $H$,
namely, $K=G^\theta$ and $K_H = H^\theta$, respectively
(throughout this paper, we use superscripts to signify the
invariant part, e.g.\ $G^\theta = \{ g \in G : \theta(g) = g \}$).
In this paper, the rank always means the complex rank
as opposed to the real rank (for instance, the rank of $\operatorname{U}(p,q)$
is not $\min \{ p,q \}$, but $p+q$), namely,
we define the rank of a reductive Lie algebra to be
the dimension of its maximal semisimple abelian subspace,
and the rank of a linear reductive Lie group to be
that of the corresponding Lie algebra.
Then, Kobayashi's conjecture is stated as follows:
\begin{conj}[{\cite[Conj.~6.4]{Kob89Geom}}]\label{rank:conj:rank}
A homogeneous space $G/H$
of reductive type does not admit a compact Clifford--Klein form if
$\rank G - \rank K < \rank H - \rank K_H$.
\end{conj}

We prove Conjecture~\ref{rank:conj:rank}
using relative Lie algebra cohomology.
Let us briefly recall its definition from a geometric viewpoint
(see Section~\ref{rank:subsect:def}
for a purely algebraic treatment).
We write $\g$, $\h$, $\k$ and $\k_H$ for the Lie algebras of
$G$, $H$, $K$ and $K_H$, respectively.
Let $H_0$ denote the identity component of $H$.
A $G$-invariant differential form on $G/H_0$
is determined by the value at $1 \cdot H_0 \in G/H_0$,
and the value must be invariant under the action of the stabilizer
$H_0$, or equivalently, of $\h$. Thus, the space
$\Omega(G/H_0)^G$ of $G$-invariant differential forms on
$G/H_0$ is naturally identified with $(\Lambda(\g/\h)^\ast)^\h$,
and the exterior differential $d$ on $G/H_0$ can be seen
as a differential on $(\Lambda(\g/\h)^\ast)^\h$.
The relative Lie algebra cohomology $H^\bl(\g, \h; \R)$
is the cohomology of the differential graded algebra
$((\Lambda(\g/\h)^\ast)^\h, d)$.

\begin{rem}\label{rank:rem:cptconn}
Suppose that $G$ is a connected compact
Lie group with Lie algebra $\g$ and
$H$ is a connected closed subgroup of $G$ with Lie algebra $\h$.
Then, the inclusion
$(\Lambda (\g/\h)^\ast)^\h \simeq \Omega(G/H)^G \injto \Omega(G/H)$
induces an isomorphism between the relative Lie algebra cohomology
$H^\bl(\g, \h; \R)$ and the de Rham cohomology $H^\bl(G/H; \R)$
(see e.g.\ \cite[Ch.~I]{Che-Eil48}).
\end{rem}

We use the following cohomological obstruction to the existence of
compact Clifford--Klein forms, which was proved in \cite{Mor15JDG}
and extended to the locally modelled case in \cite{Mor15+}.

\begin{fact}\label{rank:fact:previous}
Let $G/H$ be a homogeneous space of reductive type.
If the homomorphism
$i: H^\bl(\g, \h; \R) \to H^\bl(\g, \k_H; \R)$
induced from the inclusion
$(\Lambda (\g/\h)^\ast)^\h \injto (\Lambda (\g/\k_H)^\ast)^{\k_H}$
is not injective, then there exist no
compact manifolds locally modelled on the homogeneous space $G/H$
(and, in particular, there exist no compact Clifford--Klein forms of $G/H$).
\end{fact}

Recall that, for a reductive Lie algebra $\g$, the graded vector space
$P_{\g^\ast}$ defined by
\[
P_{\g^\ast} = \{ \alpha \in (\Lambda^+ \g^\ast)^\g :
\text{$\alpha(x \w y) = 0$ for all $x, y \in (\Lambda^+ \g)^\g$} \}
\]
is called the space of primitive elements in $(\Lambda \g^\ast)^\g$
(see Section~\ref{rank:subsect:trans}), where $\Lambda^+$
denotes the positive degree part of the exterior algebra.
We prove the following result in this paper,
which leads to the affirmative solution of
Conjecture~\ref{rank:conj:rank}.

\begin{thm}[Theorem~\ref{rank:thm:main}
$\textup{(i)} \Leftrightarrow \textup{(vii)}$]
\label{rank:thm:main_short}
Let $G/H$ be a homogeneous space of reductive type.
Then, the homomorphism
$i : H^\bl(\g, \h; \R) \to H^\bl(\g, \k_H; \R)$
is injective if and only if the linear map
$\rest: (P_{\g^\ast})^{-\theta} \to (P_{\h^\ast})^{-\theta}$
induced from the restriction map
$(\Lambda \g^\ast)^\g \to (\Lambda \h^\ast)^\h$ is surjective,
where $(\, \cdot \, )^{-\theta}$
denotes the $(-1)$-eigenspace for $\theta$.
\end{thm}

\begin{rem}[cf. Remark~\ref{rank:rem:main_rephrase}]
\label{rank:rem:main_rephrase_short}
In view of Remark~\ref{rank:rem:cptconn},
we can rephrase Theorem~\ref{rank:thm:main_short} as follows:
\textit{Let $G$ be a connected compact Lie group
with Lie algebra $\g$ and
$H$ a connected closed subgroup of $G$ with Lie algebra $\h$.
Let $\theta$ be an involution of $G$ such that $\theta(H) = H$.
Put $K_H = H^\theta$. Then, the homomorphism
$\pi^\ast : H^\bl(G/H; \R) \to H^\bl(G/K_H; \R)$
induced from the projection $\pi : G/K_H \to G/H$
is injective if and only if
the linear map
$\rest : (P_{\g^\ast})^{-\theta} \to (P_{\h^\ast})^{-\theta}$
is surjective.}
\end{rem}

Theorem~\ref{rank:thm:main_short} enables us to check easily
if the assumption of Fact~\ref{rank:fact:previous}
is satisfied or not.
Conjecture~\ref{rank:conj:rank} follows immediately from
Fact~\ref{rank:fact:previous}, Theorem~\ref{rank:thm:main_short}
and the fact that $\dim (P_{\g^\ast})^{-\theta} = \rank G - \rank K$
(Fact~\ref{rank:fact:symmetric}~(1)), as we shall explain in
Section~\ref{rank:subsect:proof_short}.

The proof of Theorem~\ref{rank:thm:main_short}
is based on the theory of H.\ Cartan, C.\ Chevalley, J.-L.\ Koszul
and A.\ Weil (\cite{Car51b}) that gives an easy way to compute
the relative Lie algebra cohomology $H^\bl(\g, \h; \R)$
of a reductive pair $(\g, \h)$.
In modern terminology of Sullivan's rational homotopy theory
(initiated by \cite{Sul77}), what they actually did is the construction
of a pure Sullivan model for the differential graded algebra
$((\Lambda(\g/\h)^\ast)^\h, d)$ from a transgression for $\g$.
By this theory, the proof is reduced to computations of
invariant polynomials and
a spectral sequence for pure Sullivan algebras.

\begin{rem}
For the proof of Conjecture~\ref{rank:conj:rank},
it is enough to show the ``only if'' part of
Theorem~\ref{rank:thm:main_short}
(i.e.\ Theorem~\ref{rank:thm:main} (i) $\Rightarrow$ (vii)).
However, we believe that Theorem~\ref{rank:thm:main_short}
itself is rather interesting in its own right,
and thus we also give the proof of the ``if'' part
(i.e.\ Theorem~\ref{rank:thm:main} (vii) $\Rightarrow$ (i))
in this paper.
\end{rem}

\begin{rem}\label{rank:rem:earlier}
Kobayashi and Ono proved Conjecture~\ref{rank:conj:rank}
in the case of $\rank G = \rank H$, investigating the Euler class
of the tangent bundle of a compact Clifford--Klein form
(\cite[Cor.~5]{Kob-Ono90}, \cite[Prop.~4.10]{Kob89Ann}).
Fact~\ref{rank:fact:previous} can be regarded as an extension of
their results to all the Chern--Weil characteristic classes
(cf.\ Theorem~\ref{rank:thm:main}
$\text{(i)} \Leftrightarrow \text{(ii)}$
and \cite[Prop.~6.1]{Mor15JDG}).
\end{rem}

\begin{rem}\label{red:rem:tholozan}
Tholozan (\cite[Ver.~2]{Tho15+}, \cite{Tho+}) independently proved
Conjecture~\ref{rank:conj:rank}.
%(a part of his proof \cite[Ver.~2, Prop.~4.3]{Tho15+}
%is postponed to his forthcoming paper).
The strategy of his proof and ours are similar;
his proof is based on a new cohomological obstruction to
the existence of compact Clifford--Klein forms,
which is a generalization of Fact~\ref{rank:fact:previous}.
It seems that his proof cannot be applied to the case of
manifolds locally modelled on $G/H$ because his new obstruction is
established only for compact Clifford--Klein forms.
However, we are not sure if it is an essential difference or not.
Indeed, as far as the author knows, a compact manifold
locally modelled on a homogeneous space of reductive type
has not been found, other than compact Clifford--Klein forms.
\end{rem}

The organization of this paper is as follows.
In Section~\ref{rank:sect:Sullivan},
we recall the definition of pure Sullivan algebras and
construct a spectral sequence arising from a homomorphism
of pure Sullivan algebras.
In Section~\ref{rank:sect:Cartan}, we recall the theory of
transgressions for a reductive Lie algebra and
the Sullivan model for a reductive pair, mostly without proof, and
apply the spectral sequence constructed in
Section~\ref{rank:sect:Sullivan} to this setting.
In Section~\ref{rank:sect:proof}, we give the proofs of
Theorem~\ref{rank:thm:main_short} and
Conjecture~\ref{rank:conj:rank}
using results in Section~\ref{rank:sect:Cartan}.

\section{Preliminaries on pure Sullivan algebras}
\label{rank:sect:Sullivan}

In this section,
we first recall the general definition of pure Sullivan algebras.
As we shall see in Section~\ref{rank:sect:Cartan},
the relative Lie algebra cohomology of a reductive pair
is computed by a certain pure Sullivan algebra.
We then construct a spectral sequence defined for
a homomorphism of pure Sullivan algebras of the form
$1 \ox g : (\Lambda U \ox S\shift{V}, -\delta_f)
\to (\Lambda U \ox S\shift{W}, -\delta_{gf})$,
which will be used in the proof of Theorem~\ref{rank:thm:main_short}
(cf.\ Theorem~\ref{rank:thm:main} (viii)).
We refer to \cite{Sul77} and \cite{FHT01}
for further results on Sullivan algebras.

Since Theorem~\ref{rank:thm:main_short}
is a purely algebraic theorem,
we work over an arbitrary field $\K$ of characteristic $0$,
rather than over $\R$, in the rest of this paper.
There are two gradings on the exterior algebra $\Lambda V$
of a graded vector space $V$, namely,
the one defined as in the ungraded case and
the one induced from the grading on $V$.
We write $\Lambda V = \bigoplus_p \Lambda^p V$ for the former grading and
$\Lambda V = \bigoplus_p (\Lambda V)^p$ for the latter.
Unless otherwise specified,
we regard $\Lambda V$ as a graded algebra by the latter grading.
We use the notation
$\Lambda^+ V$ for the positive degree part of $\Lambda V$
with respect to the former grading.
It is also the positive degree part of the latter grading
if $V$ is positively graded, which is always the case in this paper.
We define $(S V)^p$, $S^p V$ and $S^+ V$ in the same way.
Given a graded vector space $V$,
we define a new graded vector space $\shift{V}$
by $\shift{V} = V[-1]$,
i.e.\ by putting $\shift{V}^n = V^{n-1}$ for each $n \in \Z$.
We write $\shift{v}$ for the element of $\shift{V}$
corresponding to $v \in V$.
Similarly, we write $\shift{Q}$ for the element of $S\shift{V}$
corresponding to $Q \in S V$.
For $v \in V$, we denote by $\ep(v)$ and $\mu(v)$
the left multiplications by $v$ on $\Lambda V$ and $SV$,
respectively.
For $\alpha \in V^\ast$, we denote by
$\iota(\alpha)$ and $\partial(\alpha)$
the derivations of $\Lambda V$ and $SV$ uniquely determined by
$\iota(\alpha)v = \alpha(v)$ and
$\partial(\alpha)v = \alpha(v)$ ($v \in V$), respectively.
We always use the Koszul sign convention, namely, we multiply by
$(-1)^{pq}$ when we interchange two objects of homogeneous degrees
$p$ and $q$, respectively.

\subsection{Pure Sullivan algebras}\label{rank:subsect:Sullivan_def}

Let $U = \bigoplus_{n \geq 1} U^{2n-1}$
and $V = \bigoplus_{n \geq 1} V^{2n-1}$
be finite-dimensional, oddly and positively graded vector spaces.
Let $f : S\shift{U} \to S\shift{V}$
be a graded algebra homomorphism.
Define a differential $\delta_f$ on a graded algebra
$\Lambda U \ox S\shift{V}$ by the formula
\[
\delta_f = \sum_i \iota( e^i ) \ox \mu( f( \shift{e_i} ) ),
\]
where $(e_i)_i$ is a basis of $U$ and
$(e^i)_i$ the basis of $U^\ast$ dual to $(e_i)_i$.
It is called the Koszul differential associated with $f$.
In other words, the Koszul differential $\delta_f$
is the unique derivation satisfying
\[
\delta_f(u \ox 1) = 1 \ox f(\shift{u}), \quad
\delta_f(1 \ox \shift{v}) = 0 \qquad (u \in U,\ v \in V).
\]
Thus, $\delta_f$ does not depend on the choice of a basis
$(e^i)_i$, and we have $\delta_f^2 = 0$.
A differential graded algebra of the form
$( \Lambda U \ox S \shift{V}, -\delta_f )$
is called a pure Sullivan algebra.

\begin{rem}
The minus sign in our definition of a pure Sullivan algebra
is inserted just for convenience and is not essential. Indeed,
$1 \ox \sgn : (\Lambda U \ox S \shift{V}, -\delta_f)
\simto (\Lambda U \ox S \shift{V}, \delta_f)$
is an isomorphism of differential graded algebras,
where $\sgn$ denotes the automorphism of $S \shift{V}$ defined by
$\sgn|_{S^p \shift{V}} = (-1)^p$.
\end{rem}

The Koszul differential on $\Lambda V \ox S\shift{V}$
associated with the identity map $1_{S\shift{V}}$ on $S\shift{V}$
is denoted by $\delta_V$ instead of $\delta_{1_{S\shift{V}}}$.

\subsection{A spectral sequence for pure Sullivan algebras}
\label{rank:subsect:rel_Sullivan}
Let $U$, $V$ and $W$
be finite-dimensional, oddly and positively graded vector spaces.
Let $f : S\shift{U} \to S\shift{V}$ and
$g : S\shift{V} \to S\shift{W}$ be graded algebra homomorphisms.
Then,
\[
1 \ox g : (\Lambda U \ox S\shift{V}, -\delta_f)
\to (\Lambda U \ox S\shift{W}, -\delta_{gf})
\]
is a differential graded algebra homomorphism.

The Koszul differential
$\delta_f$ on $\Lambda U \ox S \shift{V}$
can be extended to the differential
$\delta_f \ox 1 \ox 1$ on
$\Lambda U \ox S \shift{V} \ox \Lambda V \ox S \shift{W}$.
By abuse of notation,
we abbreviate $\delta_f \ox 1 \ox 1$ to $\delta_f$.
Similarly, the Koszul differentials
$\delta_g$ on $\Lambda V \ox S \shift{W}$,
$\delta_{gf}$ on $\Lambda U \ox S \shift{W}$ and
$\delta_V$ on $\Lambda V \ox S\shift{V}$
are naturally extended to the differentials on
$\Lambda U \ox S \shift{V} \ox \Lambda V \ox S \shift{W}$,
which we shall denote by the same symbols.
We define a differential graded algebra homomorphism
\[
m : (\Lambda U \ox S \shift{V} \ox \Lambda V \ox S \shift{W},
-\delta_f - \delta_g + \delta_V)
\to (\Lambda U \ox S \shift{W}, -\delta_{gf})
\]
by
\begin{align*}
m(\phi \ox \shift{Q} \ox \psi \ox \shift{R}) &= 0
&(\phi \in \Lambda U,\ Q \in SV,\
\psi \in \Lambda^+ V,\ R \in SW), \\
m(\phi \ox \shift{Q} \ox 1 \ox \shift{R}) &=
\phi \ox g(\shift{Q})\shift{R}
&(\phi \in \Lambda U,\ Q \in SV,\ R \in SW).
\end{align*}
\begin{prop}\label{rank:prop:rel_Sullivan}
The homomorphism $m$ is a Sullivan model for the homomorphism
$1 \ox g : (\Lambda U \ox S\shift{V}, -\delta_f)
\to (\Lambda U \ox S\shift{W}, -\delta_{gf})$, i.e.\
\begin{enumerate}
\item[\textup{(i)}] The diagram
\[
\xymatrix{
(\Lambda U \ox S \shift{V}, -\delta_f)
\ar[r]^{1 \ox g} \ar[dr]^{i} &
(\Lambda U \ox S \shift{W}, -\delta_{gf}) \\
&
(\Lambda U \ox S \shift{V} \ox \Lambda V \ox S \shift{W},
-\delta_f - \delta_g + \delta_V) \ar[u]^{m}
}
\]
commutes, where $i$ is the natural inclusion.
\item[\textup{(ii)}] It induces an isomorphism in cohomology:
\[
m : H^\bl(\Lambda U \ox S \shift{V} \ox \Lambda V \ox S \shift{W},
-\delta_f - \delta_g + \delta_V)
\simto H^\bl(\Lambda U \ox S \shift{W}, -\delta_{gf}).
\]
\end{enumerate}
\end{prop}
\begin{rem}
The nilpotency condition on differential (\cite[p.~181]{FHT01})
is always satisfied in this situation.
\end{rem}
Proposition~\ref{rank:prop:rel_Sullivan} should be known to experts,
but we give its proof in
Section~\ref{rank:subsect:proof_rel_Sullivan}
for the sake of completeness.

Let us define a filtration $(F^p)_{p \in \N}$
of the differential graded algebra
$(\Lambda U \ox S \shift{V} \ox \Lambda V \ox S \shift{W},
-\delta_f - \delta_g + \delta_V)$
by
\[
F^p = \bigoplus_{k \geq p}
(\Lambda U \ox S \shift{V})^k \ox \Lambda V \ox S \shift{W}.
\]
The next proposition is easily obtained from
routine computations and the identification
$m : H^\bl(\Lambda U \ox S \shift{V} \ox \Lambda V \ox S \shift{W},
-\delta_f - \delta_g + \delta_V)
\simto H^\bl(\Lambda U \ox S \shift{W}, -\delta_{gf})$.
\begin{prop}\label{rank:prop:ss}
The spectral sequence $(E^{p,q}_r, d_r)$
associated with the filtration $(F^p)_{p \in \N}$
satisfies the following:
\begin{enumerate}
\item[\textup{(1)}]
$E^{p,q}_2 = H^p(\Lambda U \ox S \shift{V}, -\delta_f) \ox
H^q(\Lambda V \ox S \shift{W}, -\delta_g)$.
\item[\textup{(2)}]
The spectral sequence $(E_r^{p,q}, d_r)$ converges to
$H^{p+q}(\Lambda U \ox S \shift{W}, - \delta_{gf})$.
\item[\textup{(3)}] The homomorphism
$1 \ox g : H^p(\Lambda U \ox S \shift{V}, -\delta_f) \to
H^p(\Lambda U \ox S \shift{W}, -\delta_{gf})$
is factorized as
\[
H^p(\Lambda U \ox S \shift{V}, -\delta_f)
\simto E_2^{p,0}
\surjto E_\infty^{p,0}
\injto H^p(\Lambda U \ox S \shift{W}, -\delta_{gf}).
\]
\end{enumerate}
\end{prop}

\subsection{Proof of Proposition~\ref{rank:prop:rel_Sullivan}}
\label{rank:subsect:proof_rel_Sullivan}

The condition (i) is trivial. Let us verify the condition (ii).

For $(p,q) \in \N^2$, let $\pi_{p,q}$ denote the projection of
$\Lambda U \ox S \shift{V} \ox \Lambda V \ox S \shift{W}$ given by
\[
\pi_{p,q} = \begin{cases}
0 &\text{on} \quad
\Lambda U \ox S^{p'} \shift{V} \ox \Lambda^{q'} V \ox S \shift{W},\
(p',q') \neq (p,q), \\
1 &\text{on} \quad
\Lambda U \ox S^p \shift{V} \ox \Lambda^q V \ox S \shift{W}.
\end{cases}
\]
We write $\pi$ instead of $\pi_{0,0}$ when we regard $\pi_{0,0}$
as a map from
$\Lambda U \ox S \shift{V} \ox \Lambda V \ox S \shift{W}$ to
$\Lambda U \ox S \shift{W}$.
Define a linear endomorphism $\kappa$ of
$\Lambda U \ox S \shift{V} \ox \Lambda V \ox S \shift{W}$ by
\[
\kappa = \begin{cases}
{\displaystyle
\frac{1}{p+q}\sum_j 1 \ox \partial(\shift{f^j}) \ox \ep(f_j) \ox 1}
& \text{on} \quad \minibox{
$\Lambda U \ox S^p \shift{V} \ox \Lambda^q V \ox S \shift{W},$ \\
$(p,q) \neq (0,0),$} \\
0
& \text{on} \quad \Lambda U \ox \K \ox \K \ox S \shift{W},
\end{cases}
\]
where $(f_j)_j$ is a basis of $V$ and
$(f^j)_j$ the basis of $V^\ast$ dual to $(f_j)_j$.
One can easily show that
$\delta_V \kappa + \kappa \delta_V = 1 - \pi_{0,0}$
(see e.g.\ \cite[\S 3.1]{Gui-Ste99}). Since
\[
(\delta_g \kappa)
(\Lambda U \ox S^p \shift{V} \ox \Lambda^q V \ox S \shift{W})
\subset
\Lambda U \ox S^{p-1} \shift{V} \ox \Lambda^q V \ox S \shift{W},
\]
the infinite sum $\sum_{p=0}^\infty (\delta_g \kappa)^p$
is well-defined as a linear automorphism of
$\Lambda U \ox S \shift{V} \ox \Lambda V \ox S \shift{W}$,
whose inverse is $1-\delta_g \kappa$.
Define an endomorphism $\phi$ of the graded algebra
$\Lambda U \ox S \shift{V} \ox \Lambda V \ox S \shift{W}$ by
\begin{align*}
\phi(u \ox 1 \ox 1 \ox 1) &= u \ox 1 \ox 1 \ox 1 & \\
&\qquad + \kappa \sum_{p=0}^\infty (\delta_g \kappa)^p (1 \ox f(\shift{u}) \ox 1 \ox 1) &(u \in U), \\
\phi(1 \ox \shift{v} \ox 1 \ox 1)
&= 1 \ox \shift{v} \ox 1 \ox 1 - 1 \ox 1 \ox 1 \ox g(\shift{v}) &(v \in V), \\
\phi(1 \ox 1 \ox v \ox 1)
&= 1 \ox 1 \ox v \ox 1 &(v \in V), \\
\phi(1 \ox 1 \ox 1 \ox \shift{w})
&= 1 \ox 1 \ox 1 \ox \shift{w} &(w \in W).
\end{align*}
Then, $\phi$ has the following properties:

\begin{lem}\label{rank:lem:coord_change}
\begin{enumerate}
\item[\textup{(1)}]
$\phi (-\delta_{gf} + \delta_V)
= (-\delta_f - \delta_g + \delta_V) \phi$.
\item[\textup{(2)}]
For any
$x \in \Lambda U \ox S \shift{V} \ox \Lambda V \ox S \shift{W}$,
there exists $n \in \N$ such that $(1- \phi)^n x = 0$.
\item[\textup{(3)}]
$m\phi = \pi$.
\end{enumerate}
\end{lem}
\begin{proof}
We identify $U$, $\shift{V}$, $V$ and $\shift{W}$
as graded subspaces of
$\Lambda U \ox S \shift{V} \ox \Lambda V \ox S \shift{W}$ in a natural way.

(1). Since both sides are derivations of
$\Lambda U \ox S \shift{V} \ox \Lambda V \ox S \shift{W}$,
it suffices to verify this equality on
$U$, $\shift{V}$, $V$ and $\shift{W}$.
The only nontrivial equality is
\[
\phi (-\delta_{gf} + \delta_V) (u \ox 1 \ox 1 \ox 1)
= (-\delta_f - \delta_g + \delta_V) \phi (u \ox 1 \ox 1 \ox 1)
\qquad (u \in U).
\]
The left-hand side is equal to $-1 \ox 1 \ox 1 \ox gf(\shift{u})$,
while the right-hand side is computed as
\begin{align*}
&(-\delta_f - \delta_g + \delta_V) \phi (u \ox 1 \ox 1 \ox 1) \\
=& -1 \ox f(\shift{u}) \ox 1 \ox 1
+ (-\delta_g + \delta_V) \kappa \sum_{p=0}^\infty
(\delta_g \kappa)^p (1 \ox f(\shift{u}) \ox 1 \ox 1) \\
=& (-1 + \delta_V \kappa) \sum_{p=0}^\infty
(\delta_g \kappa)^p (1 \ox f(\shift{u}) \ox 1 \ox 1) \\
=& -(\pi_{0,0} + \kappa \delta_V) \sum_{p=0}^\infty
(\delta_g \kappa)^p (1 \ox f(\shift{u}) \ox 1 \ox 1) \\
=& -\pi_{0,0} \sum_{p=0}^\infty
(\delta_g \kappa)^p (1 \ox f(\shift{u}) \ox 1 \ox 1) \\
=& -\sum_{p=0}^\infty
(\delta_g \kappa)^p \pi_{p,0} (1 \ox f(\shift{u}) \ox 1 \ox 1).
\end{align*}
Thus, it is enough to see that
\[
(\delta_g \kappa)^p(1 \ox \shift{Q} \ox 1 \ox \shift{R})
= 1 \ox 1 \ox 1 \ox g(\shift{Q}) \shift{R}
\qquad (Q \in S^p V,\ R \in S W) \tag{$\ast_p$}
\]
holds for every $p \in \N$. Obviously ($\ast_0$) is true.
Let us assume that ($\ast_{p-1}$) is true for some $p \geq 1$.
Then, for $Q \in S^p V$ and $R \in S W$,
\begin{align*}
(\delta_g \kappa)^p(1 \ox \shift{Q} \ox 1 \ox \shift{R})
&= \frac{1}{p}(\delta_g \kappa)^{p-1}\sum_j 1 \ox
\partial(\shift{f_j})\shift{Q} \ox 1 \ox g(\shift{f^j}) \shift{R} \\
&= 1 \ox 1 \ox 1 \ox g\left( \frac{1}{p} \sum_j
\mu(\shift{f^j})\partial(\shift{f_j})\shift{Q} \right) \shift{R}
\end{align*}
by the induction hypothesis. Since
$\sum_j \mu(\shift{f^j})\partial(\shift{f_j}) = p$
on $S^p \shift{V}$, we have
\[
1 \ox 1 \ox 1 \ox g\left( \frac{1}{p} \sum_j
\mu(\shift{f^j})\partial(\shift{f_j})\shift{Q} \right) \shift{R}
= 1 \ox 1 \ox 1 \ox g(\shift{Q}) \shift{R}.
\]
Hence ($\ast_p$) is also true.
This completes the proof of Lemma~\ref{rank:lem:coord_change}~(1).

(2). Put
$A =
\{ x \in \Lambda U \ox S \shift{V} \ox \Lambda V \ox S \shift{W} :
\text{$(1- \phi)^n x = 0$ for some $n \in \N$} \}$.
Notice that $A$ is a subalgebra of
$\Lambda U \ox S \shift{V} \ox \Lambda V \ox S \shift{W}$.
Indeed, the equality
$(1-\phi)(xx') = (1-\phi)(x) x' + \phi(x) (1-\phi)(x')$
implies that, if $(1-\phi)^n x = 0$ and $(1-\phi)^{n'} x' = 0$,
then $(1-\phi)^{n+n'-1} (xx') = 0$.
Therefore, it suffices to show that
$U, \shift{V}, V, \shift{W} \subset A$.
The inclusions $\shift{V}, V, \shift{W} \subset A$ are obvious.
This implies
$\K \ox S \shift{V} \ox \Lambda V \ox S \shift{W} \subset A$.
Now, $U \subset A$ follows from
$(1-\phi)(U) \subset
\K \ox S \shift{V} \ox \Lambda V \ox S \shift{W}$.

(3). Since both sides are graded algebra homomorphisms,
it suffices to verify this equality on
$U$, $\shift{V}$, $V$ and $\shift{W}$.
The only nontrivial equality is
\[
m\phi (u \ox 1 \ox 1 \ox 1)
= \pi (u \ox 1 \ox 1 \ox 1) \qquad (u \in U),
\]
which follows from $\pi \kappa = 0$.
\end{proof}
\noindent

Now, we resume the proof of Proposition~\ref{rank:prop:rel_Sullivan}.
By Lemma~\ref{rank:lem:coord_change},
\[
\phi : (\Lambda U \ox S \shift{V} \ox \Lambda V \ox S \shift{W},
- \delta_{gf} + \delta_V)
\simto (\Lambda U \ox S \shift{V} \ox \Lambda V \ox S \shift{W},
- \delta_f - \delta_g + \delta_V)
\]
is a differential graded algebra isomorphism, whose inverse is
$\sum_{k=0}^\infty (1-\phi)^k$, that makes the diagram
\[
\xymatrix{
(\Lambda U \ox S \shift{V} \ox \Lambda V \ox S \shift{W},
-\delta_{gf} + \delta_V) \ar[d]^{\wr}_{\phi} \ar[rd]_{\pi} & \\
(\Lambda U \ox S \shift{V} \ox \Lambda V \ox S \shift{W}, -\delta_f
- \delta_g + \delta_V) \ar[r]^{\hspace{1.8cm} m} &
(\Lambda U \ox S \shift{W}, -\delta_{gf})
}
\]
commute. Thus, it suffices to show that the projection
\[
\pi : (\Lambda U \ox S \shift{V} \ox \Lambda V \ox S \shift{W},
-\delta_{gf} + \delta_V)
\to (\Lambda U \ox S \shift{W}, -\delta_{gf})
\]
induces an isomorphism in cohomology. Let
\[
i : (\Lambda U \ox S \shift{W}, -\delta_{gf})
\to (\Lambda U \ox S \shift{V} \ox \Lambda V \ox S \shift{W},
-\delta_{gf} + \delta_V)
\]
denote the natural inclusion. We have $\pi i = 1$ and
\[
i \pi = \pi_{0,0} = 1 - \delta_V \kappa - \kappa \delta_V
= 1 - (-\delta_{gf} + \delta_V) \kappa
- \kappa (-\delta_{gf} + \delta_V).
\]
Therefore,
$\pi :
H^\bl(\Lambda U \ox S \shift{V} \ox \Lambda V \ox S \shift{W},
-\delta_{gf} + \delta_V)
\to H^\bl(\Lambda U \ox S \shift{W}, -\delta_{gf})$
is an isomorphism with inverse
$i : H^\bl(\Lambda U \ox S \shift{W}, -\delta_{gf})
\to H^\bl(\Lambda U \ox S \shift{V} \ox \Lambda V \ox S \shift{W},
-\delta_{gf} + \delta_V)$.
This completes the proof of
Proposition~\ref{rank:prop:rel_Sullivan}. \qed

\section{Preliminaries on the relative Lie algebra cohomology of
reductive pairs}\label{rank:sect:Cartan}

In this section, we recall the Cartan--Chevalley--Koszul--Weil
theory (announced in \cite{Car51b})
on transgressions for a reductive Lie algebra and
the Sullivan model for a reductive pair. We mostly omit the proofs.
See \cite{GHV76} or \cite{Oni94} for details on this subject.

We retain the notations of Section~\ref{rank:sect:Sullivan}.
We always regard the dual $\g^\ast$ of a Lie algebra $\g$
as a graded vector space concentrated in degree 1.
Thus $\shift{\g^\ast}$ is concentrated in degree 2.
We write $\LL$ for the $\g$-action on the exterior algebra
$\Lambda \g^\ast$.
Given an automorphism $\theta$ of a Lie algebra $\g$,
we denote by the same symbol $\theta$ the induced automorphisms of
$(\Lambda \g^\ast)^\g$, $(S \shift{\g^\ast})^\g$, etc.
Note that our notations are not the same as any of
\cite{Car51b}, \cite{GHV76} and \cite{Oni94}; for instance,
$\Lambda P_{\g^\ast} \ox (S \shift{\h^\ast})^\h$
in our notation corresponds to
$\mathrm{I_A(G)} \ox \mathrm{I_S(\mathit{g})}$ in \cite{Car51b}, to
$(\vee \mathbb{F}^\ast)_{\theta = 0} \ox \wedge P_E$ in \cite{GHV76}
and to $C(\g, \h) = \wedge P_G \ox S_H$ in \cite{Oni94}.

\subsection{Relative Lie algebra cohomology}\label{rank:subsect:def}

Let $\g$ be a Lie algebra and $\h$ a subalgebra of $\g$.
Let $d$ be the differential on the exterior algebra
$\Lambda \g^\ast$ given by
\begin{align*}
(d\alpha)(X_1, \dots, X_{p+1}) =
\sum_{1 \leq i<j \leq p+1} \alpha([X_i, X_j], X_1, \dots,
\widehat{X_i}, \dots, \widehat{X_j}, \dots, X_{p+1})& \\
\qquad\qquad
(\alpha \in \Lambda^p \g^\ast,\ X_1, \dots, X_{p+1} \in \g)&.
\end{align*}
The graded subalgebra
\[
(\Lambda (\g/\h)^\ast)^\h = \{ \alpha \in \Lambda \g^\ast :
\text{$\iota(X)\alpha = 0$, $\LL(X)\alpha = 0$
for all $X \in \h$} \}
\]
of $\Lambda \g^\ast$ is closed under the differential $d$.
The cohomology of the differential graded algebra
$((\Lambda (\g/\h)^\ast)^\h, d)$ is denoted by $H^\bl(\g, \h; \K)$
and called the relative Lie algebra cohomology of a pair $(\g, \h)$.

\begin{rem}
If $\K = \R$ and a pair $(\g, \h)$
comes from a homogeneous space $G/H$,
it is easy to see that the above definition coincides with
the geometric definition given in Introduction.
\end{rem}

\subsection{The Cartan model of equivariant cohomology and
the Chern--Weil homomorphism}\label{rank:subsect:Cartan_model}
H. Cartan and A. Weil defined the notion of equivariant cohomology
for a differential graded algebra equipped with
``interior products'' and ``Lie derivatives''
by the elements of a Lie algebra (\cite{Car51a}, \cite{Car51b}).
We here explain their basic results in the case of
$(\Lambda \g^\ast, d)$, which admits interior products and
Lie derivatives by the elements of $\h$.
See e.g.\ \cite[Ch.~VIII]{GHV76} or \cite[\S\S 2--5]{Gui-Ste99}
for the general case.

Let $\g$ be a Lie algebra and $\h$ a subalgebra of $\g$.
Define a differential $d_{\g, \h}$ on a graded algebra
$(\Lambda \g^\ast \ox S \shift{\h^\ast})^\h$ by the formula
\[
d_{\g,\h} = d \ox 1- \sum_j \iota(F_j) \ox \mu(\shift{F^j}),
\]
where $(F_j)_j$ is a basis of $\h$ and
$(F^j)_j$ the basis of $\h^\ast$ dual to $(F_j)_j$.
The cohomology of the differential graded algebra
$((\Lambda \g^\ast \ox S \shift{\h^\ast})^\h, d_{\g,\h})$
is called the Cartan model of $\h$-equivariant cohomology of
$\Lambda \g^\ast$. The natural inclusion
\[
w : ((S \shift{\h^\ast})^\h, 0)
\to ((\Lambda \g^\ast \ox S \shift{\h^\ast})^\h, d_{\g,\h}),
\qquad \shift{Q} \mapsto 1 \ox \shift{Q}
\]
induces a homomorphism $w : (S \shift{\h^\ast})^\h
\to H^\bl((\Lambda \g^\ast \ox S \shift{\h^\ast})^\h, d_{\g,\h})$,
which is said to be the Chern--Weil homomorphism.

One has a natural inclusion of differential graded algebras
\[
\epsilon : ((\Lambda (\g/\h)^\ast)^\h, d)
\to ((\Lambda \g^\ast \ox S \shift{\h^\ast})^\h, d_{\g,\h}),
\qquad \alpha \mapsto \alpha \ox 1.
\]
\begin{fact}[cf.\ {\cite[Ch.~VIII, Th.~IV]{GHV76},
\cite[\S 5.1]{Gui-Ste99}}]\label{rank:fact:Lie_vs_Cartan}
When $\h$ has an $\h$-invariant complementary linear subspace $V$
in $\g$ (e.g.\ when $\h = \g$ or when $\h$ is reductive in $\g$),
the inclusion $\epsilon$ induces an isomorphism
$\epsilon : H^\bl(\g, \h; \K) \simto
H^\bl((\Lambda \g^\ast \ox S \shift{\h^\ast})^\h, d_{\g,\h})$.
\end{fact}
The inverse isomorphism is constructed as follows
(cf.\ \cite[Ch.~VIII, Prop.~IX]{GHV76}, \cite[\S 5.2]{Gui-Ste99}).
Let $\pi_V$ denote the projection
$\Lambda \g^\ast = \Lambda \h^\ast \ox \Lambda V^\ast
\surjto \Lambda V^\ast$.
Let $\chi : S \shift{\h^\ast} \to \Lambda V$
be the graded algebra homomorphism induced from
the graded linear map
\[
\shift{\h^\ast} \to \Lambda^2 V^\ast,
\qquad \shift{F} \mapsto -F([ \cdot, \cdot ]),
\]
where $F \in \h^\ast$ is regarded as an element of $\g^\ast$
by putting $F|_V = 0$. Then, the graded algebra homomorphism
\[
\psi_V : \Lambda \g^\ast \ox S \shift{\h^\ast}
\to \Lambda V^\ast \ (\simeq \Lambda (\g/\h)^\ast), \qquad
\alpha \ox \shift{Q} \mapsto \pi_V(\alpha) \w \chi(\shift{Q}).
\]
restricts to the differential graded algebra homomorphism
\[
\psi_V : ((\Lambda \g^\ast \ox S \shift{\h^\ast})^\h, d_{\g,\h})
\to ((\Lambda (\g/\h)^\ast)^\h, d).
\]
This $\psi_V$ induces the inverse of $\epsilon$ in cohomology.
We simply write $w$ for the composition
\[
(\epsilon^{-1} \circ w =)\, \psi_V \circ w : (S \shift{\h^\ast})^\h
\to H^\bl((\Lambda \g^\ast \ox S \shift{\h^\ast})^\h, d_{\g,\h})
\simto H^\bl(\g, \h; \K),
\]
which is also said to be the Chern--Weil homomorphism.

\begin{rem}[{cf.\ \cite[Ch.~XI, \S 1]{GHV76}}]\label{rank:rem:cptconnCW}
Recall that we can identify
$H^\bl(\g, \h; \R)$ with $H^\bl(G/H; \R)$
if $G$ is a connected compact Lie group with Lie algebra $\g$ and
$H$ is a connected closed subgroup with Lie algebra $\h$
(Remark~\ref{rank:rem:cptconn}).
Under this identification, the Chern--Weil homomorphism
$w : (S \shift{\h^\ast})^\h \to H^\bl(\g, \h; \R)$
defined here corresponds to the Chern--Weil homomorphism
$w : (S \shift{\h^\ast})^\h \to H^\bl(G/H; \R)$
for the principal $H$-bundle $G \to G/H$.
\end{rem}

\subsection{The Cartan map}

Let $\g$ be a Lie algebra.
By Fact~\ref{rank:fact:Lie_vs_Cartan}, one has
\[
H^p((\Lambda \g^\ast \ox S \shift{\g^\ast})^\g, d_{\g,\g})
\simeq H^p(\g, \g; \K)
= \begin{cases} \K &(p = 0), \\ 0 &(p \geq 1). \end{cases}
\]
Thus, for $\shift{P} \in ((S \shift{\g^\ast})^\g)^{2k} \
(= (S^k \shift{\g^\ast})^\g) \ (k \geq 1)$,
there exists a unique element
$\rho_\g(\shift{P})$ of $(\Lambda^{2k-1} \g^\ast)^\g$ such that
$d_{\g,\g}(\rho_\g(\shift{P}) \ox 1 + \Omega) = -1 \ox \shift{P}$
for some $\Omega \in (\Lambda \g^\ast \ox S^+\shift{\g^\ast})^\g$
(the uniqueness follows from $d|_{(\Lambda \g^\ast)^\g} = 0$).
This defines a linear map
$\rho_\g : (S^+ \shift{\g^\ast})^\g \to (\Lambda^+ \g^\ast)^\g$
of degree $-1$, called the Cartan map for $\g$.
See \cite[Ch.~VI, \S 2]{GHV76} for details.

\subsection{Primitive elements and transgressions}
\label{rank:subsect:trans}

Let $\g$ be a reductive Lie algebra.
Let $P_{\g^\ast}$ denote the space of primitive elements in
$(\Lambda \g^\ast)^\g$, namely,
\[
P_{\g^\ast} = \{ \alpha \in (\Lambda^+ \g^\ast)^\g :
\text{$\alpha(x \w y) = 0$ for all
$x, y \in (\Lambda^+ \g)^\g$} \}.
\]
It is known that $P_{\g^\ast}$ is oddly graded
(\cite[Ch.~V, Lem.~VII (1)]{GHV76}),
the inclusion $P_{\g^\ast} \injto (\Lambda \g^\ast)^\g$
induces an isomorphism
$\Lambda P_{\g^\ast} \simeq (\Lambda \g^\ast)^\g$
(\cite[Ch.~V, Th.~III]{GHV76})
and the dimension of $P_{\g^\ast}$ is equal to the rank of $\g$
(\cite[Ch.~X, Th.~XII]{GHV76}).

\begin{rem}
If $\g$ be a reductive Lie algebra, $(\Lambda \g^\ast)^\g$
is dual to the graded algebra $(\Lambda \g)^\g$
and therefore admits a graded coalgebra structure in a natural way.
One can easily see that $(\Lambda \g)^\g$ together with
the usual algebra structure and the above coalgebra structure
forms a graded Hopf algebra.
The above definition of $P_{\g^\ast}$ coincides with
the usual definition of the space of primitive elements in
a graded Hopf algebra.
\end{rem}

\begin{fact}[{\cite[Ch.~VI, Th.~II]{GHV76}}]
\label{rank:fact:isom_Cartan}
The Cartan map $\rho_\g$ for a reductive Lie algebra $\g$ satisfies
$\ker \rho_\g =
(S^+ \shift{\g^\ast})^\g \cdot (S^+ \shift{\g^\ast})^\g$ and
$\im \rho_\g = P_{\g^\ast}$.
\end{fact}
A linear map $\tau_\g : P_{\g^\ast} \to (S^+ \shift{\g^\ast})^\g$
of degree 1 satisfying $\rho_\g \circ \tau_\g = 1$
is called a transgression in the Weil algebra of $\g$.
We simply call it a transgression for $\g$.

\begin{fact}[{\cite[Ch.~VI, Th.~I]{GHV76}}]
\label{rank:fact:isom_transgression}
A transgression $\tau_\g$ for a reductive Lie algebra $\g$
induces a graded algebra isomorphism
$\shift{\tau_\g} : S \shift{P_{\g^\ast}} \simto
(S \shift{\g^\ast})^\ast$.
\end{fact}

The condition $\rho_\g \circ \tau_\g = 1$
is equivalent to the existence of a graded linear map
$\Omega : P_{\g^\ast} \to
(\Lambda \g^\ast \ox S^+\shift{\g^\ast})^\g$
such that
$d_{\g, \g}(\alpha \ox 1 + \Omega(\alpha))
= -1 \ox \tau_\g(\alpha) \ (\alpha \in P_{\g^\ast})$.
There exists a unique transgression $\tau_\g$ for $\g$
such that this graded linear map $\Omega$ can be taken so that
$(\iota(Z) \ox 1)(\Omega(\alpha)) = 0$ for any
$Z \in (\Lambda^+ \g)^\g$ and $\alpha \in P_{\g^\ast}$
(\cite[Ch.~VI, Prop.~VI]{GHV76}).
It is called the distinguished transgression for $\g$.

\subsection{Compatibility with automorphisms}

It is obvious from the definition of the Cartan map $\rho_\g$
for a Lie algebra $\g$ that the following diagram commutes for
any automorphism $\theta$ of $\g$:
\[
\xymatrix{
(S^+ \shift{\g^\ast})^\g \ar[r]_{\ \rho_\g} \ar[d]_{\theta} &
\Lambda^+ \g^\ast \ar[d]_{\theta}\ \\
(S^+ \shift{\g^\ast})^\g \ar[r]_{\ \rho_\g} &
\Lambda^+ \g^\ast.
}
\]
We say that a transgression $\tau_\g$ for a reductive Lie algebra
$\g$ is compatible with an automorphism $\theta$ of $\g$ if
the following diagram commutes:
\[
\xymatrix{
P_{\g^\ast} \ar[r]_{\hspace{-0.5cm} \tau_\g} \ar[d]_{\theta} &
(S^+ \shift{\g^\ast})^\g \ar[d]_{\theta}\ \\
P_{\g^\ast} \ar[r]_{\hspace{-0.5cm} \tau_\g} &
(S^+ \shift{\g^\ast})^\g.
}
\]
It readily follows from its uniqueness that
the distinguished transgression is compatible with any automorphism.

\subsection{The Sullivan model for a reductive pair}
\label{rank:subsect:standard_Sullivan}

Now, let $(\g, \h)$ be a reductive pair,
i.e.\ $\g$ a reductive Lie algebra and
$\h$ a subalgebra of $\g$ such that $\h$ is reductive in $\g$.
Let $\tau_\g : P_{\g^\ast} \to (S^+ \shift{\g^\ast})^\g$
be a transgression for $\g$ and $\shift{\tau_\g} :
S \shift{P_{\g^\ast}} \simto (S \shift{\g^\ast})^\g$
the induced isomorphism
(cf.\ Fact~\ref{rank:fact:isom_transgression}).
Define a graded algebra homomorphism
$\shift{\tau_{\g, \h}} :
S \shift{P_{\g^\ast}} \to (S \shift{\h^\ast})^\h$
by $\shift{\tau_{\g, \h}}(\shift{\Omega})
= \shift{\tau_\g}(\shift{\Omega})|_\h$.
Here, $( \cdot )|_\h :
(S \shift{\g^\ast})^\g \to (S \shift{\h^\ast})^\h$
denotes the restriction map.
We sometimes write $\rest$ instead of $( \cdot )|_\h$.
Let us consider the pure Sullivan algebra
$(\Lambda P_{\g^\ast} \ox (S\h^\ast)^\h,
-\delta_{\shift{\tau_{\g, \h}}})$
associated with $\shift{\tau_{\g,\h}}$:
\[
\delta_{\shift{\tau_{\g, \h}}} (\alpha \otimes 1)
= 1 \ox \tau_{\g}(\alpha)|_\h, \quad
\delta_{\shift{\tau_{\g, \h}}} (1 \ox \shift{Q})
= 0 \qquad (\alpha \in P_{\g^\ast},\ Q \in (S \h^\ast)^\h).
\]
By definition of $\tau_\g$, there exists a graded linear map
$\Omega : P_{\g^\ast} \to
(\Lambda \g^\ast \ox S^+\shift{\g^\ast})^\g$
such that
$d_{\g, \g}(\alpha \ox 1 + \Omega(\alpha)) = -1 \ox \tau_\g(\alpha)
\ (\alpha \in P_{\g^\ast})$.
Let us take one of such $\Omega$. The Chevalley homomorphism
\[
\vartheta_\Omega : (\Lambda P_{\g^\ast} \ox (S \shift{\h^\ast})^\h,
-\delta_{\shift{\tau_{\g,\h}}})
\to ((\Lambda \g^\ast \ox S \shift{\h^\ast})^\h, d_{\g,\h})
\]
is a differential graded algebra homomorphism defined by
\begin{align*}
\vartheta_\Omega(\alpha \ox 1)
&= \alpha \ox 1 + (1 \ox \rest)(\Omega(\alpha))
&(\alpha \in P_{\g^\ast}), \\
\vartheta_\Omega(1 \ox \shift{Q})
&= 1 \ox \shift{Q}
&(Q \in (S \h^\ast)^\h),
\end{align*}
where $\rest : S \shift{\g^\ast} \to S \shift{\h^\ast}$
is the restriction map.

\begin{fact}
[{\cite[Ch.~X, Prop.~IV]{GHV76}}]\label{rank:fact:Chevalley}
The Chevalley homomorphism $\vartheta_\Omega$
induces an isomorphism in cohomology:
\[
\vartheta_\Omega :
H^\bl(\Lambda P_{\g^\ast} \ox (S \shift{\h^\ast})^\h,
-\delta_{\shift{\tau_{\g,\h}}})
\simto H^\bl((\Lambda \g^\ast \ox S \shift{\h^\ast})^\h, d_{\g,\h})
\ (\simeq H^\bl(\g, \h; \K)).
\]
\end{fact}

\begin{rem}
Fact~\ref{rank:fact:Chevalley} means that the Chevalley homomorphism
$\vartheta_\Omega$ (resp.\ $\psi_V \circ \vartheta_\Omega$,
where $\psi_V$ is as in Section~\ref{rank:subsect:Cartan_model})
is a Sullivan model for the differential graded algebra
$((\Lambda \g^\ast \ox S \shift{\h^\ast})^\h, d_{\g,\h})$
(resp.\ $((\Lambda (\g/\h)^\ast)^\h, d)$).
We thus call it the Sullivan model for the reductive pair
$(\g, \h)$, abusing terminology.
\end{rem}

\subsection{The Chern--Weil homomorphism in the Sullivan model}

We retain the setting of
Section~\ref{rank:subsect:standard_Sullivan}. Let
$w': (S \shift{\h^\ast})^\h \to
H^\bl(\Lambda P_{\g^\ast} \ox (S \shift{\h^\ast})^\h,
-\delta_{\shift{\tau_{\g,\h}}})$
be the homomorphism induced from the inclusion
\[
w' : ((S \shift{\h^\ast})^\h, 0) \to
(\Lambda P_{\g^\ast} \ox (S \shift{\h^\ast})^\h,
-\delta_{\shift{\tau_{\g,\h}}}),
\qquad \shift{Q} \mapsto 1 \ox \shift{Q}.
\]

\begin{prop}[{\cite[Ch.~X, Prop.~IV]{GHV76}}]
\label{rank:prop:Chern-Weil-Sullivan}
The homomorphism $w'$
is identified with the Chern--Weil homomorphism
$w : (S \shift{\h^\ast})^\h
\to H^\bl((\Lambda \g^\ast \ox S \shift{\h^\ast})^\h, d_{\g,\h})
\ (\simeq H^\bl(\g, \h; \K))$
via $\vartheta_\Omega$ (or $\epsilon^{-1} \circ \vartheta_\Omega$).
\end{prop}

Indeed,
$w = \vartheta_\Omega \circ w' : ((S \shift{\h^\ast})^\h, 0)
\to ((\Lambda \g^\ast \ox S \shift{\h^\ast})^\h, d_{\g,\h})$.

\begin{prop}[{\cite[Ch.~X, Cor.~III (1) to Th.~III]{GHV76}}]
\label{rank:prop:kernel_Chern_Weil}
One has
\[
(\ker w =)\,
\ker w'= (S^+ \shift{\g^\ast})^\g|_\h \cdot (S \shift{\h^\ast})^\h.
\]
\end{prop}
This follows easily from the definition of differential
$\delta_{\shift{\tau_{\g,\h}}}$ and
Fact~\ref{rank:fact:isom_transgression}.

\subsection{The case of reductive symmetric pairs}
\label{rank:subsect:symmetric}

If $(\g, \h)$ is a reductive symmetric pair,
i.e.\ $\g$ is a reductive Lie algebra and
$\h = \g^\theta$ for some involution $\theta$ of $\g$,
the following useful results follow:

\begin{fact}\label{rank:fact:symmetric}
Let $(\g, \h)$ be a reductive symmetric pair. Then,
\begin{enumerate}
\item[\textup{(1)}]
\textup{(\cite[Ch.~X, Cor.\ to Prop.~VI]{GHV76})}
$\dim (P_{\g^\ast})^{-\theta} = \rank \g - \rank \h$.
\item[\textup{(2)}]
\textup{(\cite[Ch.~X, Prop.~VII]{GHV76})}
If $\tau_\g$ is a transgression for $\g$
that is compatible with $\theta$,
the following is a graded algebra isomorphism:
\begin{align*}
\Lambda (P_{\g^\ast})^{-\theta} \ox \im w'
&\simto H^\bl(\Lambda P_{\g^\ast} \ox (S \shift{\h^\ast})^\h, -\delta_{\shift{\tau_{\g,\h}}}), \\
\alpha \ox [1 \ox \shift{Q}] &\mapsto [\alpha \ox \shift{Q}].
\end{align*}
\end{enumerate}
\end{fact}
\begin{rem}
In \cite[Ch.~X, Prop.~VII]{GHV76},
$\tau_\g$ is assumed to be a distinguished transgression,
but its proof is, in fact,
valid for any transgression compatible with $\theta$.
\end{rem}

\subsection{Induced homomorphisms}\label{rank:subsect:induced}

Let $\g$ be a Lie algebra, $\h$ a subalgebra of $\g$ and
$\l$ a subalgebra of $\h$. Then the inclusion
\[
i :
((\Lambda (\g/\h)^\ast)^\h, d) \to ((\Lambda (\g/\l)^\ast)^\l, d)
\]
and the restriction
\[
1 \ox \rest :
((\Lambda \g^\ast \ox S \shift{\h^\ast})^\h, d_{\g, \h}) \to
((\Lambda \g^\ast \ox S \shift{\l^\ast})^\l, d_{\g, \l})
\]
are differential graded algebra homomorphisms.
The following diagram commutes:
\[
\xymatrix{
((\Lambda (\g/\h)^\ast)^\h, d) \ar[r]^{\hspace{-0.5cm} \epsilon}
\ar[d]^{i} &
((\Lambda \g^\ast \ox S \shift{\h^\ast})^\h, d_{\g,\h})
\ar[d]^{1 \ox \rest} \\
((\Lambda (\g/\l)^\ast)^\l, d) \ar[r]^{\hspace{-0.5cm} \epsilon} &
((\Lambda \g^\ast \ox S \shift{\l^\ast})^\l, d_{\g,\l})
}
\]
Suppose, in addition, that
$(\g, \h)$ and $(\g, \l)$ are reductive pairs.
Take a transgression $\tau_\g$ for $\g$ and a graded linear map
$\Omega : P_{\g^\ast} \to
(\Lambda \g^\ast \ox S^+\shift{\g^\ast})^\g$
such that
$d_{\g, \g}(\alpha \ox 1 + \Omega(\alpha))
= -1 \ox \tau_\g(\alpha)$.
Then,
\[
1 \ox \rest : (\Lambda P_{\g^\ast} \ox (S \shift{\h^\ast})^\h,
-\delta_{\shift{\tau_{\g, \h}}})
\to (\Lambda P_{\g^\ast} \ox (S \shift{\l^\ast})^\l,
-\delta_{\shift{\tau_{\g, \l}}})
\]
is a differential graded algebra homomorphism, and the diagram
\[
\xymatrix{
(\Lambda P_{\g^\ast} \ox (S \shift{\h^\ast})^\h,
-\delta_{\shift{\tau_{\g, \h}}})
\ar[r]^{\hspace{0.2cm} \vartheta_\Omega} \ar[d]^{1 \ox \rest} &
((\Lambda \g^\ast \ox S \shift{\h^\ast})^\h, d_{\g,\h})
\ar[d]^{1 \ox \rest} \\
(\Lambda P_{\g^\ast} \ox (S \shift{\l^\ast})^\l,
-\delta_{\shift{\tau_{\g, \l}}})
\ar[r]^{\hspace{0.2cm} \vartheta_\Omega} &
((\Lambda \g^\ast \ox S \shift{\l^\ast})^\l, d_{\g,\l})
}
\]
commutes. In summary,

\begin{prop}\label{rank:prop:Lie_vs_Sullivan}
The homomorphism
\[
1 \ox \rest : H^\bl(\Lambda P_{\g^\ast} \ox (S \shift{\h^\ast})^\h,
-\delta_{\shift{\tau_{\g, \h}}})
\to H^\bl(\Lambda P_{\g^\ast} \ox (S \shift{\l^\ast})^\l,
-\delta_{\shift{\tau_{\g, \l}}})
\]
is identified with the homomorphism
$i : H^\bl(\g, \h; \K) \to H^\bl(\g, \l; \K)$
via $\epsilon^{-1} \circ \vartheta_\Omega$.
\end{prop}

\subsection{A spectral sequence for the Sullivan models of
reductive pairs}\label{rank:subsect:ss2}

As in Section~\ref{rank:subsect:induced},
let $(\g, \h)$ be a reductive pair and $\l$ a subalgebra of $\h$
such that $(\g, \l)$ is a reductive pair.
Let $\tau_\g$ and $\tau_\h$ be transgressions of $\g$ and $\h$,
respectively. We identify $(S \shift{\h^\ast})^\h$ with
$S \shift{P_{\h^\ast}}$ via $\shift{\tau_\h}$.
We thus denote by $\delta_{P_{\h^\ast}}$ the Koszul differential on
$\Lambda P_{\h^\ast} \ox (S \shift{\h^\ast})^\h$ defined by
\[
\delta_{P_{\h^\ast}}(\beta \ox 1) = \tau_{\h}(\beta), \quad
\delta_{P_{\h^\ast}}(1 \ox \shift{Q}) = 0 \qquad
(\beta \in P_{\h^\ast},\ Q \in (S \h^\ast)^\h).
\]
Let us apply the spectral sequence constructed in
Section~\ref{rank:sect:Sullivan}
to the differential graded algebra homomorphism
\[
1 \ox \rest : (\Lambda P_{\g^\ast} \ox (S \shift{\h^\ast})^\h,
-\delta_{\shift{\tau_{\g, \h}}})
\to (\Lambda P_{\g^\ast} \ox (S \shift{\l^\ast})^\l,
-\delta_{\shift{\tau_{\g, \l}}}).
\]
By Proposition~\ref{rank:prop:rel_Sullivan},
the differential graded algebra homomorphism
\[
m : (\Lambda P_{\g^\ast} \ox (S \shift{\h^\ast})^\h \ox
\Lambda P_{\h^\ast} \ox (S \shift{\l^\ast})^\l,
-\delta_{\shift{\tau_{\g, \h}}} - \delta_{\shift{\tau_{\h, \l}}} +
\delta_{P_{\h^\ast}})
\to (\Lambda P_{\g^\ast} \ox (S \shift{\l^\ast})^\l,
-\delta_{\shift{\tau_{\g, \l}}})
\]
defined by
\begin{align*}
m(\alpha \ox \shift{Q} \ox \beta \ox \shift{R})
&= 0 \quad
(\alpha \in \Lambda P_{\g^\ast},\ Q \in (S \h^\ast)^\h,\
\beta \in \Lambda^+ P_{\h^\ast},\ R \in (S \l^\ast)^\l), \\
m(\alpha \ox \shift{Q} \ox 1 \ox \shift{R})
&= \alpha \ox \shift{Q}|_{\l} \cdot \shift{R}
\quad\quad \
(\alpha \in \Lambda P_{\g^\ast},\
Q \in (S \h^\ast)^\h,\ R \in (S \l^\ast)^\l)\
\end{align*}
is a Sullivan model for the differential graded algebra homomorphism
$1 \ox \rest : (\Lambda P_{\g^\ast} \ox
(S \shift{\h^\ast})^\h, -\delta_{\shift{\tau_{\g, \h}}})
\to (\Lambda P_{\g^\ast} \ox (S \shift{\l^\ast})^\l,
-\delta_{\shift{\tau_{\g, \l}}})$.
Let $(F^p)_{p \in \N}$
be a filtration of the differential graded algebra
$(\Lambda P_{\g^\ast} \ox (S \shift{\h^\ast})^\h \ox
\Lambda P_{\h^\ast} \ox (S \shift{\l^\ast})^\l,
-\delta_{\shift{\tau_{\g, \h}}} - \delta_{\shift{\tau_{\h, \l}}} +
\delta_{P_{\h^\ast}})$
defined by
$F^p = \bigoplus_{k \geq p}
(\Lambda P_{\g^\ast} \ox (S \shift{\h^\ast})^\h)^k \ox
\Lambda P_{\h^\ast} \ox (S \shift{\l^\ast})^\l$.
Applying Proposition~\ref{rank:prop:ss} to this setting,
we have the following:
\begin{cor}\label{rank:cor:ss2}
Let $(E^{p,q}_r, d_r)$ be the spectral sequence
associated with the filtration $(F^p)_{p \in \N}$. Then,
\begin{enumerate}
\item[\textup{(1)}]
$E^{p,q}_2 = H^p(\Lambda P_{\g^\ast} \ox (S \shift{\h^\ast})^\h,
-\delta_{\shift{\tau_{\g, \h}}}) \ox
H^q(\Lambda P_{\h^\ast} \ox (S \shift{\l^\ast})^\l,
-\delta_{\shift{\tau_{\h, \l}}})$.
\item[\textup{(2)}]
The spectral sequence $(E_r^{p,q}, d_r)$ converges to
$H^{p+q}(\Lambda P_{\g^\ast} \ox (S \shift{\l^\ast})^\l,
-\delta_{\shift{\tau_{\g, \l}}})$.
\item[\textup{(3)}]
The homomorphism
\[
1 \ox \rest : H^p(\Lambda P_{\g^\ast} \ox (S \shift{\h^\ast})^\h,
-\delta_{\shift{\tau_{\g, \h}}})
\to H^p(\Lambda P_{\g^\ast} \ox (S \shift{\l^\ast})^\l,
-\delta_{\shift{\tau_{\g, \l}}})
\]
is factorized as
\[
H^p(\Lambda P_{\g^\ast} \ox (S \shift{\h^\ast})^\h,
-\delta_{\shift{\tau_{\g, \h}}})
\simto E_2^{p,0}
\surjto E_\infty^{p,0}
\injto H^p(\Lambda P_{\g^\ast} \ox (S \shift{\l^\ast})^\l,
-\delta_{\shift{\tau_{\g, \l}}}).
\]
\end{enumerate}
\end{cor}

\begin{rem}\label{rank:rem:cptconnss}
Suppose that
$G$ is a connected compact Lie group with Lie algebra $\g$,
$H$ is a connected closed subgroup of $G$ with Lie algebra $\h$ and
$L$ is a connected closed subgroup of $H$ with Lie algebra $\l$.
Then, our spectral sequence may be seen as
a Sullivan model version of the Leray--Serre spectral sequence for
the fibre bundle $G/L \to G/H$ (cf.\ Remark~\ref{rank:rem:cptconn}).
\end{rem}

\section{Main theorem}\label{rank:sect:proof}
We retain the notations of Section~\ref{rank:sect:Cartan}.

\subsection{Main theorem}
Let us prove the following theorem, which gives some conditions
equivalent to the injectivity of the homomorphism
$i : H^\bl(\g, \h; \K) \to H^\bl(\g, \k_\h; \K)$.
Recall that, when $\K = \R$ and a pair $(\g, \h)$
comes from a homogeneous space $G/H$ of reductive type,
the injectivity of $i$ is a necessary condition for the existence
of a compact manifold locally modelled on $G/H$
(cf.\ Fact~\ref{rank:fact:previous}).

\begin{thm}\label{rank:thm:main}
Let $(\g, \h)$ be a reductive pair over a field $\K$
of characteristic $0$ and $\theta$ an involution of $\g$
such that $\theta(\h) = \h$.
Put $\k_\h = \h^\theta$.
Let $\tau_\g : P_{\g^\ast} \to (S \shift{\g^\ast})^\g$
be a transgression for $\g$.
Let $\tau_\h : P_{\h^\ast} \to (S \shift{\h^\ast})^\h$
be a transgression for $\h$ that is compatible with $\theta$.
Then, the following conditions are all equivalent:
\begin{enumerate}
\item[\textup{(i)}] The homomorphism
$i : H^\bl(\g, \h; \K) \to H^\bl(\g, \k_\h; \K)$
is injective.
\item[\textup{(ii)}] The homomorphism
$i|_{\im w} : \im w \to H^\bl(\g, \k_\h; \K)$
is injective, where
$w : (S \shift{\h^\ast})^\h \to H^\bl(\g, \h; \K)$
is the Chern--Weil homomorphism.
\item[\textup{(iii)}] The homomorphism
\[
1 \ox \rest : H^\bl(\Lambda P_{\g^\ast} \ox (S \shift{\h^\ast})^\h,
-\delta_{\shift{\tau_{\g, \h}}})
\to H^\bl(\Lambda P_{\g^\ast} \ox (S \shift{\k_\h^\ast})^{\k_\h},
-\delta_{\shift{\tau_{\g, \k_\h}}})
\]
is injective.
\item[\textup{(iv)}] The homomorphism
\[
(1 \ox \rest)|_{\im w'} : \im w'
\to H^\bl(\Lambda P_{\g^\ast} \ox (S \shift{\k_\h^\ast})^{\k_\h},
-\delta_{\shift{\tau_{\g, \k_\h}}})
\]
is injective, where
$w': (S \shift{\h^\ast})^\h
\to H^\bl(\Lambda P_{\g^\ast} \ox (S \shift{\h^\ast})^\h,
-\delta_{\shift{\tau_{\g,\h}}})$
 is defined by $w'(\shift{Q}) = [1 \ox \shift{Q}]$.
\item[\textup{(v)}] $((S^+ \h^\ast)^\h)^{-\theta}
\subset (S^+ \g^\ast)^\g|_\h \cdot (S \h^\ast)^\h$.
\item[\textup{(vi)}] The linear map
\begin{align*}
&\overline{\rest} :
((S^+ \g^\ast)^\g / ( (S^+ \g^\ast)^\g \cdot (S^+ \g^\ast)^\g ))^{-\theta} \\
&\qquad\qquad\qquad \to
((S^+ \h^\ast)^\h / ( (S^+ \h^\ast)^\h \cdot (S^+ \h^\ast)^\h ))^{-\theta}
\end{align*}
induced from the restriction map
$(S \g^\ast)^\g \to (S \h^\ast)^\h$ is surjective.
\item[\textup{(vii)}] The linear map
$\rest : (P_{\g^\ast})^{-\theta} \to (P_{\h^\ast})^{-\theta}$
induced from the restriction map
$(\Lambda \g^\ast)^\g \to (\Lambda \h^\ast)^\h$ is surjective.
\item[\textup{(viii)}] The spectral sequence
\begin{align*}
E^{p,q}_2
&= H^p(\Lambda P_{\g^\ast} \ox (S \shift{\h^\ast})^\h,
-\delta_{\shift{\tau_{\g, \h}}}) \ox
H^q(\Lambda P_{\h^\ast} \ox (S \shift{\k_\h^\ast})^{\k_\h},
-\delta_{\shift{\tau_{\h, \k_\h}}}) \\
& \qquad \Rightarrow
H^{p+q}(\Lambda P_{\g^\ast} \ox (S \shift{\k_\h^\ast})^{\k_\h},
-\delta_{\shift{\tau_{\g, \k_\h}}})
\end{align*}
defined as in Corollary~\ref{rank:cor:ss2}
collapses at the $E_2$-term.
\end{enumerate}
\end{thm}

\begin{rem}[cf.\ Remark~\ref{rank:rem:main_rephrase_short}]
\label{rank:rem:main_rephrase}
Suppose that
$G$ is a connected compact Lie group with Lie algebra $\g$ and
$H$ is a connected closed subgroup of $G$ with Lie algebra $\h$.
Suppose that the involution $\theta$ of $\g$
lifts to an involution of $G$ such that $\theta(H) = H$.
Put $K_H = H^{\theta}$ and let $\pi : G/K_H \to G/H$
denote the projection.
Then, the conditions (i), (ii) and (viii)
are respectively rephrased as follows:
\begin{enumerate}
\item[\textup{(i$'$)}] The homomorphism
$\pi^\ast : H^\bl(G/H; \R) \to H^\bl(G/K_H; \R)$ is injective.
\item[\textup{(ii$'$)}] The homomorphism
$\pi^\ast|_{\im w} : \im w \to H^\bl(G/K_H; \R)$ is injective, where
$w : (S \shift{\h^\ast})^\h \to H^\bl(G/H; \R)$
is the Chern--Weil homomorphism for the principal $H$-bundle
$G \to G/H$.
\item[\textup{(viii$'$)}] The Leray--Serre spectral sequence
\[
E^{p,q}_2 = H^p(G/H; \R) \ox H^q(H/K_H; \R)
\Rightarrow H^{p+q}(G/K_H; \R)
\]
for the fibre bundle $\pi : G/K_H \to G/H$
collapses at the $E_2$-term.
\end{enumerate}
(cf.\ Remarks~\ref{rank:rem:cptconn}, \ref{rank:rem:cptconnCW} and
\ref{rank:rem:cptconnss}).
Theorem~\ref{rank:thm:main}
says that these conditions are all equivalent,
and they are also equivalent to the algebraic conditions (v)--(vii).
\end{rem}
\begin{proof}[Proof of Theorem~\ref{rank:thm:main}]
(i) $\Rightarrow$ (ii).
Trivial.

(iii) $\Rightarrow$ (iv).
Trivial.

(i) $\Leftrightarrow$ (iii).
This follows from Proposition~\ref{rank:prop:Lie_vs_Sullivan}.

(ii) $\Leftrightarrow$ (iv).
This follows from Propositions~\ref{rank:prop:Chern-Weil-Sullivan}
and \ref{rank:prop:Lie_vs_Sullivan}.

(iv) $\Rightarrow$ (v).
Take any $Q \in ((S \h^\ast)^\h)^{-\theta}$. Then we have
$Q|_{\k_\h} = 0$. By (iv), $[1 \ox \shift{Q}] = 0$ in
$H^\bl(\Lambda P_{\g^\ast} \ox (S \shift{\h^\ast})^\h,
-\delta_{\shift{\tau_{\g,\h}}})$.
This means $Q \in (S^+ \g^\ast)^\g|_\h \cdot (S \h^\ast)^\h$
by Proposition~\ref{rank:prop:kernel_Chern_Weil}.

(v) $\Rightarrow$ (vi).
Take any
$\overline{Q} \in ( (S^+ \h^\ast)^\h /
( (S^+ \h^\ast)^\h \cdot (S^+ \h^\ast)^\h ) )^{-\theta}$.
Let $Q \in (S^+ \h^\ast)^\h$ be a representative of $\overline{Q}$.
By (v), we can write
\[
\frac{Q - \theta(Q)}{2} = P|_\h + \sum_{i=1}^r P_i|_\h \cdot Q_i
\qquad (P, P_i \in (S^+ \g^\ast)^\g,\ Q_i \in (S^+ \h^\ast)^\h).
\]
Put $P' = (P - \theta(P))/2$. Then $\overline{P'} \in
( (S^+ \g^\ast)^\g /
( (S^+ \g^\ast)^\g \cdot (S^+ \g^\ast)^\g ) )^{-\theta}$
and $\overline{P'|_\h} = \overline{Q}$.

(vi) $\Rightarrow$ (v).
We shall prove
\[
((S^n \h^\ast)^\h)^{-\theta}
\subset (S^+ \g^\ast)^\g|_\h \cdot (S \h^\ast)^\h \tag{$\dagger_n$}
\]
by induction on $n$.
Assume that ($\dagger_m$) is true for $m \leq n-1$.
Let us take any $Q \in ((S^n \h^\ast)^\h)^{-\theta}$.
By (vi), we can write
\begin{align*}
Q = P|_\h + \sum_{i=1}^r Q_i \cdot Q'_i
\qquad (&P \in (S^n \g^\ast)^\g,\ Q_i \in (S^{m_i}\h^\ast)^\h, \\
&Q'_i \in (S^{n-m_i} \h^\ast)^\h,\ 1 \leq m_i \leq n-1)
\end{align*}
Then,
\begin{align*}
Q = \frac{1}{2} (Q - \theta(Q))
= \frac{1}{2}(P - \theta(P))|_\h
+ \frac{1}{4}
\sum_{i=1}^r ( &(Q_i - \theta(Q_i))(Q'_i + \theta(Q'_i)) \\
\quad  &+ (Q_i + \theta(Q_i))(Q'_i - \theta(Q'_i)) ).
\end{align*}
We have
\[
Q_i - \theta(Q_i),\ Q'_i - \theta(Q'_i)
\in (S^+ \g^\ast)^\g|_\h \cdot (S \h^\ast)^\h
\]
by the induction hypothesis, and therefore
$Q \in (S^+ \g^\ast)^\g|_\h \cdot (S \h^\ast)^\h$.
Thus ($\dagger_n$) is also true.

(vi) $\Leftrightarrow$ (vii).
This follows from commutativity of the diagram
\[
\xymatrix{
\left( (S^+ \shift{\g^\ast})^\g / ( (S^+ \shift{\g^\ast})^\g \cdot
(S^+ \shift{\g^\ast})^\g ) \right)^{-\theta}
\ar[r]^{\qquad\qquad\quad\ \sim}_{\qquad\qquad\quad\
\overline{\rho_\g}} \ar[d]_{\overline{\rest}} &
(P_{\g^\ast})^{-\theta} \ar[d]^{\rest} \\
\left( (S^+ \shift{\h^\ast})^\h / ( (S^+ \shift{\h^\ast})^\h \cdot
(S^+ \shift{\h^\ast})^\h ) \right)^{-\theta}
\ar[r]^{\qquad\qquad\quad\ \sim}_{\qquad\qquad\quad\
\overline{\rho_\h}} &
(P_{\h^\ast})^{-\theta},
}
\]
where $\overline{\rho_\g}$ and $\overline{\rho_\h}$
are the linear isomorphisms induced from the Cartan maps.

(v) $\Rightarrow$ (viii).
We shall prove $d_r = 0 \ (r \geq 2)$ by induction on $r$.
Let us assume that $d_s = 0$ for $2 \leq s \leq r-1$. Then
\[
E_r^{p,q} = E_2^{p,q}
= H^p(\Lambda P_{\g^\ast} \ox (S \shift{\h^\ast})^\h,
-\delta_{\shift{\tau_{\g, \h}}}) \ox
H^q(\Lambda P_{\h^\ast} \ox (S \shift{\k_\h^\ast})^{\k_\h},
-\delta_{\shift{\tau_{\h, \k_\h}}}).
\]
By Leibniz's rule, to prove $d_r = 0$,
it suffices to see that $d_r|_{E_r^{0,q}} = 0$ for all $q \geq 0$.
Moreover, by Fact~\ref{rank:fact:symmetric}~(2)
and again by Leibniz's rule, it is enough to prove that
\begin{itemize}
\item $d_r([1 \ox 1] \ox [1 \ox \shift{R}]) = 0$
for any $R \in (S \k_\h^\ast)^{\k_\h}$.
\item $d_r([1 \ox 1] \ox [\beta \ox 1]) = 0$
for any $\beta \in (P_{\h^\ast})^{-\theta}$.
\end{itemize}
By construction of the spectral sequence, we have
$d_r([1 \ox 1] \ox [1 \ox \shift{R}]) = 0$ and
\[
d_r([1 \ox 1] \ox [\beta \ox 1]) =
\begin{cases}
[1 \ox \tau_\h(\beta)] \ox [1 \ox 1]
&\text{if $\beta \in (P^{r-1}_{\h^\ast})^{-\theta}$}, \\
0
&\text{if $\beta \in (P^q_{\h^\ast})^{-\theta},\ q \neq r-1$}.
\end{cases}
\]
Since $\tau_\h$ is taken to be compatible with $\theta$,
it follows that
$\tau_\h(\beta) \in ((S \shift{\h^\ast})^\h)^{-\theta}$.
By (v), we have $\tau_\h(\beta)
\in (S^+ \shift{\g^\ast})^\g|_\h \cdot (S \shift{\h^\ast})^\h$.
This implies that $[1 \ox \tau_\h(\beta)] = 0$ in
$H^\bl(\Lambda P_{\g^\ast} \ox (S \shift{\h^\ast})^\h,
-\delta_{\shift{\tau_{\g, \h}}})$
by Proposition~\ref{rank:prop:kernel_Chern_Weil}.
We have thus proved $d_r = 0$.

(viii) $\Rightarrow$ (iii).
This follows immediately from Corollary~\ref{rank:cor:ss2}~(3).
\end{proof}

\subsection{Proof of Conjecture~\ref{rank:conj:rank}}
\label{rank:subsect:proof_short}

Suppose that the inequality
$\rank G - \rank K < \rank H - \rank K_H$
holds. Then, the linear map
$\rest : (P_{\g^\ast})^{-\theta} \to (P_{\h^\ast})^{-\theta}$
cannot be surjective because
$\dim {(P_{\g^\ast})^{-\theta}} = \rank G - \rank K$ and
$\dim {(P_{\h^\ast})^{-\theta}} = \rank H - \rank K_H$
(Fact~\ref{rank:fact:symmetric}~(1)).
Applying Theorem~\ref{rank:thm:main}
$\text{(i)} \Rightarrow \text{(vii)}$
and Fact~\ref{rank:fact:previous}, we conclude the nonexistence of
compact manifolds locally modelled on $G/H$ and, in particular,
of compact Clifford--Klein forms of $G/H$. \qed

\begin{acknow}
The author expresses his sincere thanks to Toshiyuki Kobayashi
for his valuable advice and warm encouragement.
Thanks are also due to an anonymous referee for
comments and suggestions that improved the presentation of
this paper.
This work was supported by JSPS KAKENHI Grant Numbers 14J08233 and
17H06784, and the Program for Leading Graduate Schools, MEXT, Japan.
\end{acknow}

\noindent
\textsc{Department of Mathematics, Graduate School of Science, \\
Kyoto University, \\
Kitashirakawa Oiwake-cho, Sakyo-ku, Kyoto 606-8502, Japan} \\
\textit{Email address}: \texttt{yosuke.m@math.kyoto-u.ac.jp}

\end{document}